\documentclass[11pt]{article}
\usepackage{amsmath,amssymb}

\def\N{\mathbb{N}}

\def\R{\mathbb{R}}

\def\C{C^{\infty}(M,\mathbb{R})}
\def\c{C^{\infty}(M,\mathbb{C})}

\newtheorem{definition}{Definition}[section]

\newtheorem{proposition}[definition]{Proposition}
\newtheorem{theorem}[definition]{Theorem}
\newtheorem{remark}[definition]{Remark}

\newenvironment{proof}{\noindent{\bf Proof.}}{\hfill $\square$}

\begin{document}

\title{On the geometric quantization of twisted Poisson manifolds}
\author{Fani Petalidou\\ Faculty of Sciences and Technology
\\University of Peloponnese \\22100 Tripoli - Greece\\ \vspace{5mm} {\small E-mail: petalido@uop.gr}}

\date{}
\maketitle

\begin{abstract}
We study the geometric quantization process for twisted Poisson
manifolds. First, we introduce the notion of Lichnerowicz-twisted
Poisson cohomology for twisted Poisson manifolds and we use it in
order to characterize their prequantization bundles and to
establish their prequantization condition. Next, we introduce a
polarization and we discuss the quantization problem. In each
step, several examples are presented.
\end{abstract}

\vspace{3mm}


\vspace{3mm}

\noindent {\bf{A.M.S. classification (2000):}} 53D50, 53D17.

\section{Introduction}
\emph{Geometric quantization} is a useful procedure, founded in
differential geometry, that allows us to understand the relation
between classical and quantum mechanics by associating a quantum
system to each classical system. This process consists of
attaching to each classical system a complex Hilbert space and to
each classical observable on the phase space of the classical
system a quantum observable, i.e., a Hermitian operator on the
Hilbert space, in such a way that the Poisson bracket of two
classical observables is attached, up to a purely imaginary
constant, with the commutator of the operators. It is completed in
two steps: \emph{(i) the prequantization} and \emph{(ii) the
quantization}. If $M$ is the phase space of the classical system
equipped with a symplectic structure $\omega$, at the first step,
one associates to $M$ a Hermitian line bundle $\pi : K \to M$ with
a Hermitian connection having as curvature form the symplectic
form $\omega$. $K$ is called the \emph{prequantization bundle} of
$(M,\omega)$ and exists under the prequantization condition:
\emph{The cohomology class of $\omega$ is integral}. Then, the
Poisson Lie algebra $(\C, \{\cdot,\cdot\})$ acts faithfully on the
space of cross sections $\Gamma(K)$ of $\pi : K \to M$. At the
second step, imposing a \emph{polarization}, one constructs the
Hilbert space $\mathcal{H}$ used in quantum mechanics out of
$\Gamma(K)$ and one restricts the problem to a suitable Lie
subalgebra of $(\C, \{\cdot,\cdot\})$ that is represented
irreducibly on $\mathcal{H}$. For a short introductive
presentation of the subject, we can consult \cite{mrl-q}. For a
more extensive, but without too much detail, treatment of the
problem, we suggest \cite{bw} where we can find a complete guide
to the literature. We also refer, as standard references, the
books \cite{sn} and \cite{wd}.

The theory of geometric quantization was first developed for
symplectic manifolds by B. Kostant \cite{kst} and J.M. Souriau
\cite{sou}, independently. Their approaches are different, but
equivalent, and they have important applications. Later, it was
extended by J. Huebschmann \cite{hbs} to Poisson algebras and by
I. Vaisman \cite{vai} (see, also \cite{vai-b}) to Poisson
manifolds. In \cite{hbs}, the geometric quantization of Poisson
manifolds appears as a particular case of the geometric
quantization of Poisson algebras, while, in \cite{vai}, this
quantization is presented straightforwardly, using usual
differential geometric techniques. In \cite{lmp}, Kostant's theory
was adapted by M. de Le\'{o}n \textit{et al}. for Jacobi manifolds
and, recently, by A. Weinstein and M. Zambon \cite{wz} for Dirac
manifolds.

The purpose of the present paper is to study the geometric
quantization problem for \emph{twisted Poisson manifolds}. A such
manifold $M$ is equipped with a bivector field $\Lambda$ of which
the Schouten bracket with itself is equal to the image by
$\Lambda^{\sharp}$ of a closed $3$-form $\varphi$ on $M$. These
manifolds were introduced by P. \v{S}evera and A. Weinstein in
\cite{sw}, under the name \emph{Poisson manifolds with $3$-form
background}, stimulated by the works of J.S. Park \cite{p}, L.
Cornalba and R. Schiappa \cite{cs}, C. Klim\v{c}\'{i}k and T.
Str\"{o}bl \cite{kl} on deformation quantization and string theory
in which such $3$-forms played an important role. In order to
understand the role of $\varphi$ on a twisted Poisson manifold
$(M,\Lambda)$, one introduces on the space $\C$ of the real smooth
functions on $M$ the bracket $\{f,g\}=\Lambda(df,dg)$ and one
looks its Jacobi identity which is true up to an extra term
involving $\varphi$. Thus, $(\C, \{\cdot,\cdot\})$ is no longer a
Lie algebra. This result has an essential influence on the
prequantization procedure of a twisted Poisson manifold as is
explained in paragraph \ref{sec-prequant}.

The paper is organized as follows. In section \ref{sec-tp} we
recall the definition of a twisted Poisson manifold, we give some
main examples of such manifolds and we introduce the
Lichnerowicz-twisted Poisson cohomology. In section
\ref{sec-tp-chern}, the notion of twisted Poisson-Chern class of a
complex line bundle over a twisted Poisson manifold is defined by
using the concept of contravariant derivative given by I. Vaisman
in \cite{vai}. Section \ref{sec-prequant} is devoted to the
formulation of the integrality prequantization condition of a
twisted Poisson manifold. Several interesting examples are
discussed. Finally, in section \ref{sec-quant}, we develop the
quantization process of a twisted Poisson manifold by introducing
a polarization and we present a computational example.

We mention that the deformation quantization of twisted Poisson
structures is discussed in the papers \cite{s} and \cite{abjs}
with physical motivation. Also, we note that A. Weinstein and P.
Xu developed in \cite{wx} an alternative approach to the
quantization problem of Poisson manifolds by using symplectic
groupoids. We believe that we can extend their method to twisted
Poisson manifolds by using twisted symplectic groupoids that are
introduced in \cite{cx} by A. Cattaneo and P. Xu. We postpone this
study to a future paper.

\section{Twisted Poisson manifolds}\label{sec-tp}
A \emph{twisted Poisson manifold} is a differentiable manifold $M$
equipped with a bivector field $\Lambda$ and a closed $3$-form
$\varphi$ on $M$, called the \emph{background $3$-form}, such that
\begin{equation}\label{def-tw-Poisson}
\frac{1}{2}[\Lambda,\Lambda]=\Lambda^{\sharp}(\varphi).
\end{equation}
In the above formula, $[\cdot,\cdot]$ denotes the Schouten bracket
and $\Lambda^{\sharp}$ is the natural extension of $\Lambda^{\sharp}
: \Gamma(T^*M)\to \Gamma(TM)$, given, for all $\alpha,\beta \in
\Gamma(T^*M)$, by
\begin{equation}\label{bnd-map}
\langle \beta, \Lambda^{\sharp}(\alpha)\rangle =
\Lambda(\alpha,\beta),
\end{equation}
to a homomorphism from $\Gamma(\bigwedge^kT^*M)$ to
$\Gamma(\bigwedge^kTM)$, $k\in \N$, defined, for all $\eta \in
\Gamma(\bigwedge^k T^*M)$ and $\alpha_1,\ldots,\alpha_k \in
\Gamma(T^*M)$, by
\begin{equation}\label{formule-homo}
\Lambda^{\sharp}(\eta)(\alpha_1,\ldots,\alpha_k ) =
(-1)^k\eta(\Lambda^{\sharp}(\alpha_1),\ldots,\Lambda^{\sharp}(\alpha_k
))
\end{equation}
and, for any $f\in \C$, by $\Lambda^{\sharp}(f)=f$. In the
following, a twisted Poisson manifold will be denoted by the
triple $(M,\Lambda,\varphi)$.

\subsection{Examples of twisted Poisson manifolds}\label{examples-tw-Poisson}
\emph{1) Poisson manifolds:} Let $(M,\Lambda)$ be a Poisson
manifold, i.e., $[\Lambda,\Lambda]=0$, and $\varphi$ a closed
$3$-form on $M$ satisfying $\Lambda^{\sharp}(\varphi)=0$. Then,
$(M,\Lambda,\varphi)$ is a twisted Poisson manifold. This happens
for $3$-dimensional Poisson manifolds. Since
$\mathrm{Im}\Lambda^{\sharp}$ defines a foliation of $M$ whose the
leaves are of dimension 0 or 2, we have that any three sections of
$\mathrm{Im}\Lambda^{\sharp}$ are linearly dependent on $M$. Thus,
any $3$-form $\varphi$ on $M$ is closed and
$\Lambda^{\sharp}(\varphi)=0$.

\vspace{2mm}

\noindent \emph{2) Twisted Poisson manifolds associated to
symplectic manifolds I:} Let $(M_0,\omega_0)$ be a symplectic
manifold of dimension $2n$, $n\geq 2$, and $\Lambda_0$ the unique
bivector field on $M_0$ given, for all $\alpha\in \Gamma(T^*M_0)$,
by $i(\Lambda_0^{\sharp}(\alpha))\omega_0 = -\alpha$, i.e.,
$\Lambda_0=\Lambda_0^{\sharp}(\omega_0)$. Then, for any non
constant function $f\in C^{\infty}(M_0,\mathbb{R})$, the bivector
field $\Lambda = f\Lambda_0$ and the closed $3$-form $\varphi =
-f^{-2}\omega_0\wedge df$ define a twisted Poisson structure on
$M_0$. In fact, by a simple computation, we find
\begin{eqnarray*}
\frac{1}{2}[\Lambda,\Lambda] & = &
\frac{1}{2}[f\Lambda_0,f\Lambda_0] = -(f\Lambda_0)\wedge
\Lambda_0^{\sharp}(df) \\
& = & (f\Lambda_0)^{\sharp}(-f^{-2}\omega_0\wedge
df)=\Lambda^{\sharp}(\varphi).
\end{eqnarray*}

\vspace{2mm}

\noindent \emph{3) Twisted Poisson manifolds associated to
symplectic manifolds II:} Let $(M_0,\omega_0)$ be a
$2n$-dimensional symplectic manifold with $n\geq 2$ and
$\Lambda_0$ the nondegenerate Poisson structure defined by
$\omega_0$ as in Example 2. Then, the triple
$(M,\Lambda,\varphi)$, where $M=M_0\times \R$,
$$
\Lambda = e^t(\Lambda_0 + \Lambda_0^{\sharp}(df)\wedge
\frac{\partial}{\partial t}) \quad \mathrm{and} \quad \varphi =
-e^{-t}\omega_0\wedge dt,
$$
$t$ being the canonical coordinate on $\R$ and $f\in
C^{\infty}(M_0,\R)$, is a twisted Poisson manifold. We have
\begin{eqnarray*}
\frac{1}{2}[\Lambda,\Lambda] & = & \frac{1}{2}[e^t(\Lambda_0 +
\Lambda_0^{\sharp}(df)\wedge \frac{\partial}{\partial t}),\,
e^t(\Lambda_0 + \Lambda_0^{\sharp}(df)\wedge
\frac{\partial}{\partial t})] \\
& = & e^{2t}\Lambda_0^{\sharp}(df)\wedge \Lambda_0 =
\Lambda^{\sharp}(-e^{-t}\omega_0\wedge dt) =
\Lambda^{\sharp}(\varphi).
\end{eqnarray*}

\vspace{2mm}

\noindent \emph{4) Twisted Poisson manifolds associated to Poisson
manifolds:} Let $(M,\Lambda_0,\omega)$ be a Poisson manifold
endowed with a $2$-form $\omega$ such that the operator $Id +
\omega^{\flat}\circ \Lambda_0^{\sharp} : T^*M \to T^*M$ is
invertible. Then, the vector bundle map $\Lambda^{\sharp} =
\Lambda_0^{\sharp}\circ (Id + \omega^{\flat}\circ
\Lambda_0^{\sharp})^{-1}$ defines a $(-d\omega)$-twisted Poisson
structure on $M$. (For more details, see \cite{sw}.)

\vspace{2mm}

\noindent \emph{5) Twisted Poisson structures induced by twisted
Jacobi manifolds:} Let $(M,\Lambda,E,\omega)$ be a twisted Jacobi
manifold (\cite{jf}), i.e., $M$ is a differentiable manifold
endowed with a bivector field $\Lambda$, a vector field $E$ and a
$2$-form $\omega$ such that
\begin{equation}\label{eq-jac-1}
\frac{1}{2}[\Lambda,\Lambda] +E \wedge \Lambda  =
\Lambda^{\sharp}(d\omega)+ ( \Lambda^{\sharp} \omega)\wedge E
\end{equation}
and
\begin{equation} \label{eq-jac-2}
[E,\Lambda]=  (\Lambda^{\sharp} \otimes 1)(d\omega)(E)- ((
\Lambda^{\sharp} \otimes 1)(\omega)(E)) \wedge E.
\end{equation}
In  (\ref{eq-jac-2}), $(\Lambda^{\sharp} \otimes 1)(d\omega)$ and
$(\Lambda^{\sharp} \otimes 1)(\omega)$ denote, respectively, the
sections of $(\bigwedge ^2 TM) \otimes T^*M$ and $ TM \otimes T^*M$
that act on multivector fields by contraction with the factor in
$T^*M$ (see, \cite{jf}). We consider a submanifold $M_0$ of $M$, of
codimension $1$ and transverse to $E$. Let $\varpi : U \to M_0$ be
the projection on $M_0$ of a tubular neighbourhood $U$ of $M_0$ in
$M$ such that, for any $x\in M_0$, $\varpi^{-1}(x)$ is a connected
arc of the integral curve of $E$ through $x$. If $\omega =
\varpi^{*}\omega_0$, where $\omega_0$ is a $2$-form on $M_0$, then,
the twisted Jacobi structure $(\Lambda,E,\omega)$ of $M$ induces a
twisted Poisson structure $(\Lambda_0,\varphi_0)$ on $M_0$, where
$\Lambda_0 = \varpi_*\Lambda$ and $\varphi_0 = d\omega_0$. In fact,
by projecting (\ref{eq-jac-1}) along the integral curves of $E$, we
get
$$
\frac{1}{2}[\Lambda_0,\Lambda_0] = \Lambda_0^{\sharp}(d\omega_0),
$$
while the projection of (\ref{eq-jac-2}) is annihilated identically.
(For more details, see \cite{pt}).

\subsection{The Lichnerowicz-twisted Poisson cohomology of a twisted Poisson manifold}
Let $(M,\Lambda,\varphi)$ be a twisted Poisson manifold.  As in
the case of a Poisson manifold, we introduce hamiltonian vector
fields on $M$ by setting, for any $f\in \C$,
$X_f=\Lambda^{\sharp}(df)$ and we define on $\C$ the internal
composition law
$$
\{f,g\}=\Lambda(df,dg), \hspace{5mm} f,g\in \C,
$$
that is bilinear and skew-symmetric but its Jacobi identity is
modified by $\varphi$:
$$
\{f,\{g,h\}\} +
\{g,\{h,f\}\}+\{h,\{f,g\}\}=\Lambda^{\sharp}(\varphi)(df,dg,dh).
$$
Therefore, $(\C,\{\cdot,\cdot\})$ is no longer a Lie algebra. In
this paper, we will say that it is a $\varphi$-\emph{twisted Lie
algebra}. Since the Jacobi identity is violated, we cannot, in
general, define the \emph{Chevalley-Eilenberg} cohomology of
$(\C,\{\cdot,\cdot\})$ relative to the representation defined by
the hamiltonian vector fields, i.e., to the representation given
by
$$
\C \times \C \to \C, \hspace{5mm} (f,g)\to X_f(g).
$$

However, any twisted Poisson structure $(\Lambda,\varphi)$ on $M$
produces a Lie algebroid structure on the cotangent bundle $T^*M$
of $M$, as in the ordinary case. The Lie bracket on the space of
smooth sections of $T^*M$ is given, for any $\alpha,\beta \in
\Gamma(T^*M)$, by
\begin{equation}\label{def-Lie-br}
\{\alpha,\beta\}^{\varphi}=\{\alpha,\beta\}+\varphi
(\Lambda^{\sharp}(\alpha),\Lambda^{\sharp}(\beta),\cdot),
\end{equation}
where $\{\cdot,\cdot\}$ denotes the Koszul bracket (\cite{kz})
associated to $\Lambda$, i.e.
\begin{equation}\label{def-kz}
\{\alpha,\beta\} = \mathcal{L}_{\Lambda^{\sharp}(\alpha)}\beta -
\mathcal{L}_{\Lambda^{\sharp}(\beta)}\alpha -
d\Lambda(\alpha,\beta),
\end{equation}
and characterized by $\{df,dg\}=d\{f,g\}$ and the Leibniz identity
$\{\alpha,f\beta\}=f\{\alpha,\beta\}+(\mathcal{L}_{\Lambda^{\sharp}(\alpha)}f)\beta$.
The anchor map is the vector bundle map $\Lambda^{\sharp} : T^*M \to
TM$ defined by (\ref{bnd-map}), while, the exterior derivative
operator $\partial_{\varphi}$ on $\Gamma(\bigwedge TM)$ determined
by $(\{\cdot,\cdot\}^{\varphi}, \Lambda^{\sharp})$ is defined, for
all $P\in \Gamma(\bigwedge^k TM)$ and $\alpha_1,\ldots,\alpha_{k+1}
\in \Gamma(T^*M)$, by
\begin{eqnarray*}
\partial_{\varphi}P(\alpha_1,\ldots,\alpha_{k+1}) & = &
\Sigma_{i=1}^{k+1}(-1)^{i+1}\Lambda^{\sharp}(\alpha_i)(
P(\alpha_1,\ldots,\hat{\alpha}_i,\ldots,\alpha_{k+1})) \\
&  & +\, \Sigma_{1\leq i<j\leq k+1}(-1)^{i+j} P
(\{\alpha_i,\alpha_j\}^{\varphi}, \alpha_1,\ldots,
\hat{\alpha}_i,\ldots,\hat{\alpha}_j,\ldots,\alpha_{k+1}),
\end{eqnarray*}
where the hat denotes missing arguments. Since
$\partial_{\varphi}^2=0$, $(\Gamma(\bigwedge
TM),\partial_{\varphi})$ is a chain complex.

\begin{definition} We call \emph{Lichnerowicz-twisted Poisson cohomology (L-tP
cohomology)} of $(M,\Lambda,\varphi)$ the cohomology of
$(\Gamma(\bigwedge TM),\partial_{\varphi})$. It is denoted by
$H_{L-tP}^*(M,\Lambda,\varphi)$ or, for simplicity, $H_{L-tP}^*(M)$
and, for any $k\in \N$,
$$
H_{L-tP}^k(M) = \frac{\mathrm{ker} (\partial_{\varphi} : \bigwedge^k
TM \to \bigwedge^{k+1}TM)}{\mathrm{Im} (\partial_{\varphi} :
\bigwedge^{k-1} TM \to \bigwedge^k TM)},
$$
with the convention $\bigwedge^{-1}TM=\{0\}$. The cohomology class
of any element $P\in \mathrm{ker} (\partial_{\varphi} : \bigwedge^k
TM \to \bigwedge^{k+1}TM)$ is denoted $[P]^{\varphi}$.
\end{definition}

By a simple, but long, computation, we can prove that the
homomorphism $\Lambda^{\sharp} : \Gamma(\bigwedge^* T^*M) \to
\Gamma(\bigwedge^* TM)$ is a chain map, namely,
\begin{equation}\label{type-ch-der}
\partial_{\varphi}\circ \Lambda^{\sharp} = -\Lambda^{\sharp}\circ d.
\end{equation}
Hence, we deduce
\begin{proposition}
If $H_{dR}^*(M,\R)$ is the de Rham cohomology of
$(M,\Lambda,\varphi)$, the homomorphism of complexes
$\Lambda^{\sharp} : (\Gamma(\bigwedge^* T^*M),d) \to
(\Gamma(\bigwedge^* TM),\partial_{\varphi})$ induces a
homomorphism in cohomology, also denoted by $\Lambda^{\sharp}$,
\begin{equation}\label{homo-coho}
\begin{array}{lccc}
\Lambda^{\sharp} : &  H_{dR}^*(M,\R)  & \rightarrow & H_{L-tP}^*(M) \\
& [\alpha] & \mapsto & [\Lambda^{\sharp}(\alpha)]^{\varphi}.
\end{array}
\end{equation}
If $\Lambda$ is nondegenerate, then (\ref{homo-coho}) is an
isomorphism.
\end{proposition}

\section{Twisted Poisson-Chern class of a complex line bundle over a twisted Poisson
manifold}\label{sec-tp-chern}
Let $(M,\Lambda,\varphi)$ be a
twisted Poisson manifold, $\pi : K\to M$ a complex line bundle
over $M$, $\Gamma(K)$ the space of the global cross sections of
$\pi : K\to M$ and $\mathrm{End}_{\mathbb{C}}(\Gamma(K))$ the
space of the complex linear endomorphisms of $\Gamma(K)$.

\begin{definition}
A \emph{contravariant derivative} $D$ on $\pi : K\to M$ is a
$\R$-linear mapping
$$
D : \Gamma(T^*M) \to \mathrm{End}_{\mathbb{C}}(\Gamma(K)),
$$
i.e., for any $\alpha,\beta \in \Gamma(T^*M)$ and $f\in \C$,
\begin{equation}\label{lin-D}
D_{\alpha + \beta} = D_{\alpha} + D_{\beta} \hspace{5mm}
\mathrm{and} \hspace{5mm} D_{f\alpha}=fD_{\alpha},
\end{equation}
such that
\begin{equation}\label{rel-contr-der}
D_{\alpha}(fs)=f D_{\alpha}s + (\Lambda^{\sharp}(\alpha)f)s,
\hspace{5mm} \mathrm{for}\; \mathrm{all}\;\, s\in \Gamma(K).
\end{equation}
\end{definition}

We say that $D$ is \emph{Hermitian} or \emph{compatible with a
Hermitian metric} $h$ on $\pi : K\to M$, if, for all $\alpha\in
\Gamma(T^*M)$ and $s_1,s_2 \in \Gamma(K)$,
\begin{equation}\label{def-hermitian}
\Lambda^{\sharp}(\alpha)(h(s_1,s_2)) = h(D_{\alpha}s_1,s_2) +
h(s_1,D_{\alpha}s_2).
\end{equation}
We note that such (Hermitian) operators on $\pi : K\to M$ always
exist; it suffices to consider  an arbitrary (Hermitian) connection
$\nabla$ on $\pi : K\to M$ and to put
$D_{\alpha}=\nabla_{\Lambda^{\sharp}(\alpha)}$.

\begin{definition}
The \emph{curvature} $C_D$ of a contravariant derivative $D$ on $\pi
: K\to M$ is the mapping
$$
C_D : \Gamma(T^*M) \times \Gamma(T^*M)  \to
\mathrm{End}_{\mathbb{C}}(\Gamma(K))
$$
defined, for all $\alpha, \beta \in\Gamma(T^*M)$, by
\begin{equation}\label{eq-curv}
C_D(\alpha,\beta) = D_{\alpha}\circ D_{\beta} - D_{\beta}\circ
D_{\alpha}-D_{\{\alpha,\beta\}^{\varphi}}.
\end{equation}
\end{definition}

\begin{proposition}
$C_D$ is bilinear over $\C$ and skew-symmetric, i.e.,
$$
C_D(\alpha,\beta) = - C_D(\beta, \alpha), \hspace{6mm}
\mathrm{for}\; \mathrm{all}\; \,\alpha,\beta \in \Gamma(T^*M).
$$
\end{proposition}
\begin{proof} The skew-symmetry of $C_D$ is
an immediate consequence of its definition (\ref{eq-curv}). Its
bilinearity can be proved by using the linearity (\ref{lin-D}) and
the property (\ref{rel-contr-der}) of $D$.
\end{proof}

\vspace{3mm}

Thus, from the above results and the fact that $\pi : K \to M$ is
a complex line bundle over $M$, we have that there exists a
globally defined complex bivector field $\Pi = \Pi_1 + i \Pi_2$ on
$M$, with $\Pi_1, \Pi_2 \in \Gamma(\bigwedge^2 TM)$, such that,
for all $\alpha,\beta \in \Gamma(T^*M)$ and $s \in \Gamma(K)$,
\begin{equation}\label{rel-curv-Pi}
C_D(\alpha,\beta)(s) = \Pi(\alpha,\beta)s.
\end{equation}
For more details, we can consult \cite{kst} and adapt its results
in the contravariant framework.

\vspace{3mm}

We extend, by linearity, the cohomology operator
$\partial_{\varphi}$ on the complex multivector fields on $M$ by
setting, for any $P \in \Gamma(\bigwedge^k T_{\mathbb{C}}M)$, $P=P_1
+ iP_2$ with $P_1,P_2 \in \Gamma(\bigwedge^k TM)$,
$$
\partial_{\varphi}P = \partial_{\varphi}P_1 + i
\partial_{\varphi}P_2.
$$
Clearly, $\partial_{\varphi}^2 = 0$. Consequently,
$(\Gamma(\bigwedge T_{\mathbb{C}}M),\partial_{\varphi})$ is a chain
complex whose cohomology will be called the \emph{complex
Lichnerowicz-twisted Poisson cohomology of} $(M,\Lambda,\varphi)$
and will be denoted by $H_{_{\mathbb{C}}L-tP}^*(M,\Lambda,\varphi)$
or $H_{_{\mathbb{C}}L-tP}^*(M)$.

\begin{theorem}\label{three-proper}
Let $\pi : K \to M$ be a complex line bundle over a twisted Poisson
manifold $(M,\Lambda,\varphi)$, $D$ a contravariant derivative on
$\pi : K \to M$, $C_D$ the curvature of $D$ and $\Pi$ the complex
bivector field on $M$ associated to $C_D$ (\ref{rel-curv-Pi}). Then:
\begin{enumerate}
\item [(i)] $\Pi$ defines a cohomology class $[\Pi]^{\varphi}$ in
$H_{_{\mathbb{C}}L-tP}^2(M)$. \item[(ii)] $[\Pi]^{\varphi}$ does
not depend of the contravariant derivative $D$. \item[(iii)] In
the case where $D$ is compatible with a Hermitian metric $h$ on
$\pi : K \to M$, $\Pi$ is purely imaginary.
\end{enumerate}
\end{theorem}
\begin{proof}
(i) Let $s$ be a nowhere vanishing local section of $\pi : K \to M$.
Since the complex dimension of the fibre of $\pi : K \to M$ is $1$,
we may associate to $s$ a unique complex local vector field on $M$
as follows. It is clear that, for any $1$-form $\alpha$ on $M$,
$\frac{D_{\alpha}s}{s}$ is a complex function on $M$ and the
application $\alpha \mapsto \frac{D_{\alpha}s}{s}$ is
$\mathbb{C}$-linear (\ref{lin-D}). Hence, there exists a unique
complex local vector field $X = X_1+iX_2$ on $M$, with $X_1,X_2$
local real vector fields on $M$, such that, for all $\alpha \in
\Gamma(T^*M)$,
\begin{equation}\label{D-X}
D_{\alpha}s = \langle \alpha, X\rangle s.
\end{equation}
We have that
\begin{equation}\label{Pi-X}
\Pi = \partial_{\varphi}X.
\end{equation}
Effectively, for all $\alpha,\beta \in \Gamma(T^*M)$,
\begin{eqnarray*}
\Pi(\alpha,\beta)s & \stackrel{(\ref{rel-curv-Pi})}{=} &
C_D(\alpha,\beta)(s) \\
&\stackrel{(\ref{eq-curv})}{=} & (D_{\alpha}\circ D_{\beta} -
D_{\beta}\circ D_{\alpha}-D_{\{\alpha,\beta\}^{\varphi}})(s) \\
& \stackrel{(\ref{D-X})}{=} & D_{\alpha}(\langle\beta,X\rangle s) -
D_{\beta}(\langle\alpha,X\rangle s) - \langle
\{\alpha,\beta\}^{\varphi},X\rangle s\\
& \stackrel{(\ref{rel-contr-der})(\ref{D-X})}{=} &
\langle\beta,X\rangle \langle \alpha,X\rangle s +
\Lambda^{\sharp}(\alpha)(\langle\beta,X\rangle) s \\
&  & - \,\langle\alpha,X\rangle \langle \beta,X\rangle s -
\Lambda^{\sharp}(\beta)(\langle \alpha,X\rangle) s - \langle
\{\alpha,\beta\}^{\varphi},X\rangle s \\
& = & \partial_{\varphi}X(\alpha,\beta)s.
\end{eqnarray*}
Consequently, $\partial_{\varphi}\Pi = \partial_{\varphi}^2X =0$
which means that $\Pi$ defines a cohomology class in
$H_{_{\mathbb{C}}L-tP}^2(M)$ denoted by $[\Pi]^{\varphi}$.

\noindent (ii) Let $\tilde{D}$ be another contravariant derivative
on $\pi : K \to M$ having curvature $C_{\tilde{D}}$ and $\tilde{X}$
the corresponding local complex vector field (see, (i)). We denote
by $\tilde{\Pi}$ the corresponding to $C_{\tilde{D}}$ global complex
bivector field on $M$ (\ref{rel-curv-Pi}). From (\ref{Pi-X}), we
obtain
\begin{equation}\label{difference-Pi}
\tilde{\Pi} - \Pi = \partial_{\varphi}\tilde{X} -
\partial_{\varphi}X \Leftrightarrow \tilde{\Pi} = \Pi +
\partial_{\varphi}(\tilde{X}-X).
\end{equation}
Now, for any $\alpha \in \Gamma(T^*M)$, we define the mapping
$$
\hat{D}_{\alpha} = \tilde{D}_{\alpha} - D_{\alpha} : \Gamma(K) \to
\Gamma(K)
$$
that is $\mathbb{C}$-linear. Therefore, there exists a globally
defined complex vector field $\hat{X}$ on $M$ such that, for all
$s\in \Gamma(K)$,
$$
\hat{D}_{\alpha}s = \langle \alpha, \hat{X}\rangle s.
$$
From the last two relations, we deduce that, in the overlapping of
$X$ and $\tilde{X}$,
\begin{equation}\label{rel-X}
\hat{X} = \tilde{X}-X.
\end{equation}
So, using (\ref{rel-X}) in (\ref{difference-Pi}), we obtain
$\tilde{\Pi} = \Pi +
\partial_{\varphi}\hat{X}$, which means that
$[\tilde{\Pi}]^{\varphi}=[\Pi]^{\varphi}$.

\noindent (iii) We assume that $D$ is compatible with a Hermitian
metric $h$ on $\pi : K \to M$ and let $(e)$ be a local orthonormal
basis of $\Gamma(K)$. Then, for all $\alpha \in \Gamma(T^*M)$,
(\ref{def-hermitian}) gives us
$$
\Lambda^{\sharp}(\alpha)(h(e,e)) = h(D_{\alpha}e,e) +
h(e,D_{\alpha}e) \Leftrightarrow 0 \stackrel{(\ref{D-X})}{=}
h(\langle \alpha,X\rangle e,e) + h(e,\langle \alpha,X\rangle e)
\Leftrightarrow
$$
$$
0 = \langle \alpha,X\rangle + \overline{\langle \alpha,X\rangle}
\Leftrightarrow 0 = X + \bar{X},
$$
where the bar denotes complex conjugation. Hence, $X$ is purely
imaginary and, because $\Pi = \partial_{\varphi}X$, we conclude that
$\Pi$ is purely imaginary.
\end{proof}

\vspace{3mm}

From the above theorem we get the following definition.

\begin{definition}\label{tw-Poisson-Chern-class}
Let $\pi : K \to M$ be a complex line bundle over a twisted Poisson
manifold $(M,\Lambda,\varphi)$, $D$ a contravariant derivative on
$\pi : K \to M$ having curvature $C_D$ whose the associated bivector
field $\Pi$ is purely imaginary. Then, the well-defined cohomology
class $[\frac{i}{2\pi}\Pi]^{\varphi} \in H_{L-tP}^2(M)$ will be
called the \emph{first real twisted Poisson-Chern class of} $\pi : K
\to M$.
\end{definition}

Next, we will prove that $[\frac{i}{2\pi}\Pi]^{\varphi}$ is the
image by the homomorphism (\ref{homo-coho}) of the usual first real
Chern class of $\pi : K \to M$.

\vspace{2mm}

We recall that, given a complex Hermitian line bundle $\pi : K \to
M$ over a smooth manifold $M$, the \emph{first real Chern class
of} $\pi : K \to M$ is an element of the second de Rham cohomology
of $M$ with integer coefficients and it is denoted $c_1(K,\R)$,
\cite{kbn}. On the other hand, if $\pi : K \to M$ is endowed with
a Hermitian connection $\nabla$ with curvature $C_{\nabla}$, i.e.,
for all $X,Y \in \Gamma(TM)$,
$$
C_{\nabla}(X,Y) = \nabla_X\circ \nabla_Y - \nabla_Y \circ \nabla_X -
\nabla_{[X,Y]},
$$
there exists a purely imaginary closed $2$-form $\Omega$ on $M$
(\cite{kbn}) such that, for all $s\in \Gamma(K)$,
\begin{equation}\label{conn-Omega}
C_{\nabla}(X,Y)(s)=\Omega(X,Y)s
\end{equation}
and, in this case, the first real Chern class $c_1(K,\R)$ of $\pi :
K \to M$ is just (\cite{kst}) the integral cohomology class
$[\frac{i}{2\pi}\Omega]$ in $H_{dR}^2(M,\R)$. We note that
(\cite{kst}) the canonical injection $\varepsilon : \mathbb{Z} \to
\R$ induces a homomorphism
$$
\varepsilon : H_{dR}^2(M,\mathbb{Z}) \to H_{dR}^2(M,\R)
$$
and a class $[\alpha] \in H_{dR}^2(M,\R)$ is called \emph{integral}
if it lies in the image $\mathrm{Im}\varepsilon$ of $\varepsilon$.

\begin{theorem}\label{th-Chern-class}
Let $\pi : K \to M$ be a complex Hermitian line bundle over a
twisted Poisson manifold $(M,\Lambda,\varphi)$, $\nabla$ a Hermitian
connection on $\pi : K \to M$ and $D$ the associated to $\nabla$
Hermitian contravariant derivative on $\pi : K \to M$, i.e., for any
$\alpha \in \Gamma(T^*M)$,
$D_{\alpha}=\nabla_{\Lambda^{\sharp}(\alpha)}$. If $c_1(K,\R)$ and
$[\frac{i}{2\pi}\Pi]^{\varphi}$ are, respectively, the first real
Chern class and the first real twisted Poisson-Chern class of $\pi :
K \to M$, then
\begin{equation*}
\Lambda^{\sharp}(c_1(K,\R)) = [\frac{i}{2\pi}\Pi]^{\varphi},
\end{equation*}
where $\Lambda^{\sharp} : H_{dR}^2(M,\R) \to H_{L-tP}^2(M)$ is the
homomorphism (\ref{homo-coho}) between the second de Rham cohomology
and the corresponding Lichnerowicz-twisted Poisson cohomology of $M$.
\end{theorem}
\begin{proof}
Let $\omega$ be the local, purely imaginary, connection $1$-form on
$M$ associated to $\nabla$ (\cite{kst}) as follows. For any nowhere
vanishing local section $s$ of $\pi : K \to M$ and any $Y \in
\Gamma(TM)$,
\begin{equation}\label{omega}
\nabla_Ys = \langle \omega,Y\rangle s.
\end{equation}
Then, the purely imaginary closed $2$-form $\Omega$ on $M$
associated to $C_{\nabla}$ (\ref{conn-Omega}) coincides with
$d\omega$ (see, \cite{kbn}) and $c_1(K,\R)=[\frac{i}{2\pi}\Omega] =
[\frac{i}{2\pi}d\omega]$. Moreover, if $X$ is the local purely
imaginary vector field on $M$ defined by (\ref{D-X}), from the
definition of $D$, we get that, for any $\alpha \in \Gamma(T^*M)$,
\begin{eqnarray}\label{omega-X}
D_{\alpha}s = \nabla_{\Lambda^{\sharp}(\alpha)}s &
\stackrel{(\ref{D-X})(\ref{omega})}{\Leftrightarrow} & \langle
\alpha,X \rangle s = \langle \omega, \Lambda^{\sharp}(\alpha)\rangle s \nonumber \\
& \Leftrightarrow & \langle \alpha,X \rangle s = - \langle \alpha,
\Lambda^{\sharp}(\omega)\rangle s \nonumber \\
& \Leftrightarrow & X = - \Lambda^{\sharp}(\omega).
\end{eqnarray}
Thus, if $\Pi$ is the purely imaginary bivector field on $M$
associated to the curvature $C_D$ of $D$ (\ref{rel-curv-Pi}), we
have
$$
\Pi \stackrel{(\ref{Pi-X})}{=}\partial_{\varphi}X
\stackrel{(\ref{omega-X})}{=} -
\partial_{\varphi}\Lambda^{\sharp}(\omega)\stackrel{(\ref{type-ch-der})}{=}\Lambda^{\sharp}(d\omega).
$$
Consequently,
$$
[\frac{i}{2\pi}\Pi]^{\varphi} =
[\frac{i}{2\pi}\Lambda^{\sharp}(d\omega)]^{\varphi}
\stackrel{(\ref{homo-coho})}{=}\Lambda^{\sharp}([\frac{i}{2\pi}d\omega])
=\Lambda^{\sharp}(c_1(K,\R)).
$$
\end{proof}

\section{Prequantization of twisted Poisson manifolds}\label{sec-prequant}
In this section, we will prequantize a twisted Poisson manifold
$(M,\Lambda,\varphi)$ by associating to each differentiable
function on $M$ an operator that acts on the space of cross
sections of a Hermitian line bundle $\pi : K \to M$. As we have
mentioned in Introduction, this approach was first developed by B.
Kostant \cite{kst} and J.M. Souriau \cite{sou} for symplectic
manifolds and was extended by J. Huebschmann \cite{hbs} and I.
Vaisman \cite{vai} to Poisson manifolds, by M. de Le\'{o}n
\textit{et al}. \cite{lmp} to Jacobi manifolds and by A. Weinstein
and M. Zambon \cite{wz} to Dirac manifolds.

\vspace{3mm}

Let $(M,\Lambda,\varphi)$ be a twisted Poisson manifold and $\pi : K
\to M$ a Hermitian line bundle over $M$ endowed with a contravariant
derivative $D$ whose curvature is $C_D$. We define a representation
$\,\widehat{}\,$ of the $\varphi$-twisted Lie algebra
$(\C,\{\cdot,\cdot\})$ on $\mathrm{End}_{\mathbb{C}}(\Gamma(K))$ by
associating to each $f\in \C$ a complex endomorphism $\hat{f}$ of
$\Gamma(K)$ that is defined, for any $s\in \Gamma(K)$, by
\begin{equation}\label{def-repres}
\hat{f}(s) = D_{df}s + 2\pi i fs.
\end{equation}
Since $(\C,\{\cdot,\cdot\})$ is not a Lie algebra, the map
$$
\begin{array}{lccc}
\widehat{} \; : &  \C  & \rightarrow & \mathrm{End}_{\mathbb{C}}(\Gamma(K)) \\
& f & \mapsto & \hat{f}
\end{array}
$$
is no longer a homomorphism between $(\C,\{\cdot,\cdot\})$ and
$(\mathrm{End}_{\mathbb{C}}(\Gamma(K)), [\cdot,\cdot])$, where
$[\cdot,\cdot]$ denotes the usual commutator on
$\mathrm{End}_{\mathbb{C}}(\Gamma(K))$, as the prequantization
process requires. For this reason, we consider the subspace
$$
A =\{\hat{f}\in \mathrm{End}_{\mathbb{C}}(\Gamma(K)) \, / \, f \in
\C\}
$$
of $\mathrm{End}_{\mathbb{C}}(\Gamma(K))$ and define on this the
bracket
\begin{equation}\label{phi-bracket}
\hspace{30mm}[\hat{f},\hat{g}]^{\varphi}=[\hat{f},\hat{g}]-D_{\varphi(\Lambda^{\sharp}(df),\Lambda^{\sharp}(dg),\cdot)},
\quad \quad \quad \hat{f},\hat{g}\in A,
\end{equation}
where $[\hat{f},\hat{g}]= \hat{f}\circ \hat{g} -
\hat{g}\circ\hat{f}$, in order to obtain a faithful representation
of $(\C,\{\cdot,\cdot\})$ on $(A, [\cdot,\cdot]^{\varphi})$.

\begin{proposition}\label{br-curv}
The representation $\;\,\widehat{}\, : (\C, \{\cdot,\cdot\}) \to
(A, [\cdot,\cdot]^{\varphi})$ is a homomorphism, i.e., for all
$f,g \in \C$,
\begin{equation}\label{homo-bracket}
\widehat{\{f,g\}} = [\hat{f},\hat{g}]^{\varphi},
\end{equation}
if, and only if,
\begin{equation}\label{rel-curv-br}
C_D(df,dg) = -2\pi i \{f,g\}.
\end{equation}
\end{proposition}
\begin{proof}
By a simple computation, using (\ref{def-repres}) and
(\ref{rel-contr-der}), we get
\begin{equation}\label{br-commut}
[\hat{f},\hat{g}]  =  \hat{f}\circ \hat{g} - \hat{g}\circ\hat{f} =
D_{df}\circ D_{dg} - D_{dg}\circ D_{df} + 4\pi i \{f,g\}.
\end{equation}
On the other hand, we have
\begin{eqnarray*}
\widehat{\{f,g\}} & \stackrel{(\ref{def-repres})}{=} & D_{d\{f,g\}}
+ 2\pi i \{f,g\} \nonumber \\
& \stackrel{(\ref{def-Lie-br})(\ref{lin-D})}{=} &
D_{\{df,dg\}^{\varphi}} -
D_{\varphi(\Lambda^{\sharp}(df),\Lambda^{\sharp}(dg),\cdot)} +
2\pi i
\{f,g\} \nonumber \\
& \stackrel{(\ref{eq-curv})}{=} & D_{df}\circ D_{dg} - D_{dg}\circ
D_{df} - C_D(df,dg)  \nonumber \\
&  & -\,D_{\varphi(\Lambda^{\sharp}(df),\Lambda^{\sharp}(dg),\cdot)}
+ 4\pi i
\{f,g\} - 2\pi i \{f,g\} \nonumber \\
& \stackrel{(\ref{br-commut})(\ref{phi-bracket})}{=} &
[\hat{f},\hat{g}]^{\varphi} - C_D(df,dg)- 2\pi i \{f,g\}.
\end{eqnarray*}
Thus, (\ref{homo-bracket}) holds if, and only if,
(\ref{rel-curv-br}) holds.
\end{proof}

\begin{definition}\label{def-pre-quan}
We say that a twisted Poisson manifold $(M,\Lambda,\varphi)$ is
\emph{prequantizable} if there exists a Hermitian complex line
bundle $\pi : K \to M$,  \emph{the prequantization bundle}, such
that the operators (\ref{def-repres}) make sense on $\Gamma(K)$ and
satisfy (\ref{homo-bracket}).
\end{definition}

Hence, according to Proposition \ref{br-curv} and the above
Definition, the prequantization problem of a twisted Poisson
manifold $(M,\Lambda,\varphi)$ has a solution if, and only if,
there exists a Hermitian complex line bundle $\pi : K \to M$
equipped with a contravariant derivative $D$ whose the curvature
$C_D$ satisfies
\begin{equation}\label{curv-Lambda}
C_D = -2\pi i \Lambda.
\end{equation}
We see that $C_D$ must be purely imaginary, fact that obliges us
to consider $D$ compatible with the Hermitian structure of $\pi :
K \to M$.

\begin{theorem}
A twisted Poisson manifold $(M,\Lambda,\varphi)$ is prequantizable
if, and only if, there exist a vector field $Z$ on $M$ and a
closed $2$-form $\Phi$ on $M$, which represents an integral
cohomology class of $M$, such that the following relation holds on
$M$:
\begin{equation}\label{cond-quant}
\Lambda + \partial_{\varphi}Z = \Lambda^{\sharp}(\Phi).
\end{equation}
\end{theorem}
\begin{proof}
We consider that $(M,\Lambda,\varphi)$ is prequantizable. Then,
there exists a Hermitian complex line bundle $\pi : K \to M$ with
a Hermitian contravariant derivative $D$ whose curvature $C_D$
verifies (\ref{curv-Lambda}), consequently
\begin{equation}\label{rel-Lambda-Pi}
\Lambda = \frac{i}{2\pi}C_D
\stackrel{(\ref{rel-curv-Pi})}{=}\frac{i}{2\pi}\Pi,
\end{equation}
where $\Pi$ is the purely imaginary, $\partial_{\varphi}$-closed,
bivector field on $M$ associated to $C_D$. On the other hand, let
$\nabla$ be a Hermitian connection on $\pi : K \to M$ with
curvature $2$-form $\Omega$, i.e., for all $X,Y \in \Gamma(TM)$
and $s\in \Gamma(K)$, $C_{\nabla}(X,Y)(s) = \Omega(X,Y)s$, that is
purely imaginary and closed. So, $\Phi = \frac{i}{2\pi}\Omega$ is
a real closed $2$-form on $M$ and represents the first real Chern
class $c_1(K,\R)$ of $\pi : K \to M$ which is integral, i.e.,
$c_1(K,\R)=[\Phi]$ (see, section 3). Now, we consider the
Hermitian contravariant derivative $\bar{D}$ on $\pi : K \to M$
defined by $\nabla$, i.e., for any $\alpha \in \Gamma(T^*M)$,
$\bar{D}_{\alpha} = \nabla_{\Lambda^{\sharp}(\alpha)}$. Let
$\bar{\Pi}$ be the purely imaginary bivector field on $M$
associated to $C_{\bar{D}}$ as in (\ref{rel-curv-Pi}). According
to Theorem \ref{th-Chern-class}, we have $\Lambda^{\sharp}([\Phi])
= [\frac{i}{2\pi}\bar{\Pi}]^{\varphi}
\stackrel{(\ref{homo-coho})}{\Leftrightarrow}[\Lambda^{\sharp}(\Phi)]^{\varphi}
= [\frac{i}{2\pi}\bar{\Pi}]^{\varphi}$. But, property (iii) of
Theorem \ref{three-proper} yields
$[\bar{\Pi}]^{\varphi}=[\Pi]^{\phi}$, which means that there
exists a purely imaginary vector field $W$ on $M$ such that
$\bar{\Pi} = \Pi + \partial_{\varphi}W$. Hence,
$$
\frac{i}{2\pi}\bar{\Pi} = \frac{i}{2\pi}\Pi +
\frac{i}{2\pi}\partial_{\varphi}W \Leftrightarrow
\Lambda^{\sharp}(\Phi) = \Lambda + \partial_{\varphi}Z,
$$
where $Z=\frac{i}{2\pi}W$.

Conversely, we assume that there exist a vector field $Z$ and a
closed $2$-form $\Phi$ on $(M,\Lambda,\varphi)$ such that
(\ref{cond-quant}) is true on $M$. Then, there exists a Hermitian
complex line bundle $\pi : K \to M$ over $M$ equipped with a
Hermitian connection $\nabla$ having as curvature $2$-form the
purely imaginary closed $2$-form $-2\pi i \Phi$. Using $\nabla$,
we define a contravariant derivative $D : \Gamma(T^*M) \to
\mathrm{End}_{\mathbb{C}}(\Gamma(K))$ on $\pi : K \to M$ as
follows: for all $\alpha \in \Gamma(T^*M)$ and $s\in \Gamma(K)$,
\begin{equation}\label{D-connection}
D_{\alpha}s = \nabla_{\Lambda^{\sharp}(\alpha)}s + 2\pi i \langle
\alpha, Z\rangle s.
\end{equation}
By a straightforward computation, we can prove that $D$ is
Hermitian. Also, we have that its curvature $C_{\bar{D}}$ satisfies
(\ref{curv-Lambda}). In fact, for all $\alpha, \beta \in
\Gamma(T^*M)$ and $s\in \Gamma(K)$,
\begin{eqnarray*}
C_D(\alpha,\beta)(s) & \stackrel{(\ref{eq-curv})}{=} &
(D_{\alpha}\circ D_{\beta} - D_{\beta}\circ
D_{\alpha}-D_{\{\alpha,\beta\}^{\varphi}})(s) \\
& \stackrel{(\ref{D-connection})}{=} &
D_{\alpha}(\nabla_{\Lambda^{\sharp}(\beta)}s + 2\pi i \langle
\beta, Z\rangle s) - D_{\beta}(\nabla_{\Lambda^{\sharp}(\alpha)}s
+ 2\pi i
\langle \alpha, Z\rangle s) \\
&  & - \,\nabla_{\Lambda^{\sharp}(\{\alpha,\beta\}^{\varphi})}s -
2\pi i
\langle \{\alpha,\beta\}^{\varphi}, Z\rangle s \\
& \stackrel{(\ref{D-connection})}{=} &
\nabla_{\Lambda^{\sharp}(\alpha)}(\nabla_{\Lambda^{\sharp}(\beta)}s
+ 2\pi i \langle \beta, Z\rangle s) + 2\pi i \langle \alpha,
Z\rangle (\nabla_{\Lambda^{\sharp}(\beta)}s + 2\pi i \langle
\beta, Z\rangle s)
\\
&  & -
\,\nabla_{\Lambda^{\sharp}(\beta)}(\nabla_{\Lambda^{\sharp}(\alpha)}s
+ 2\pi i \langle \alpha, Z\rangle s) - 2\pi i \langle \beta,
Z\rangle (\nabla_{\Lambda^{\sharp}(\alpha)}s + 2\pi i \langle
\alpha, Z\rangle s)
\\
&  & -\,\nabla_{[\Lambda^{\sharp}(\alpha),\Lambda^{\sharp}(\beta)]}s
- 2\pi i
\langle \{\alpha,\beta\}^{\varphi}, Z\rangle s \\
& = &
C_{\nabla}(\Lambda^{\sharp}(\alpha),\Lambda^{\sharp}(\beta))s +
2\pi i(\Lambda^{\sharp}(\alpha)\langle \beta, Z\rangle -
\Lambda^{\sharp}(\beta)\langle \alpha, Z\rangle \\
&  & -\, \langle\{\alpha,\beta\}^{\varphi}, Z \rangle )s \\
& = & - 2\pi i
\Phi(\Lambda^{\sharp}(\alpha),\Lambda^{\sharp}(\beta))s +
2\pi i \partial_{\varphi}Z (\alpha, \beta)s \\
& \stackrel{(\ref{formule-homo})(\ref{cond-quant})}{=} & - 2\pi i
\Lambda(\alpha, \beta)s,
\end{eqnarray*}
whence we conclude that $(M,\Lambda, \varphi)$ is prequantizable.
\end{proof}

\begin{remark}
{\rm Since the first Chern class of a complex line bundle over a
differentiable manifold $M$ is a complete invariant used to classify
complex line bundles over $M$, i.e., there is a bijection between
the isomorphism classes of complex line bundles over $M$ and the
elements of $H^2_{dR}(M,\mathbb{Z})$ (\cite{kbn}), we have that $K$
is not unique. Any other Hermitian complex line bundle over $M$
isomorphic to $K$ can be viewed as a prequantization bundle of
$(M,\Lambda,\varphi)$.}
\end{remark}

\subsection{Examples}\label{ex-prequan-tw-Poisson}
\emph{1) Poisson manifolds:} Let $(M,\Lambda, \varphi)$ be a
twisted Poisson manifold such that $\Lambda^{\sharp}(\varphi)=0$,
i.e. $(M,\Lambda)$ is a Poisson manifold. Then, the cotangent
bundle $T^*M$ of $M$ is equipped with two different Lie algebroids
structures $(\{\cdot,\cdot\}, \Lambda^{\sharp})$ and
$(\{\cdot,\cdot\}^{\varphi}, \Lambda^{\sharp})$ whose the brackets
are given, respectively, by (\ref{def-kz}) and (\ref{def-Lie-br}).
If $D$ is a contravariant derivative on an Hermitian complex line
bundle $\pi : K\to M$ over $M$, then its curvatures $R_D$ and
$C_D$ with respect to $\{\cdot,\cdot\}$ and
$\{\cdot,\cdot\}^{\varphi}$, respectively, are related, for any
$\alpha, \beta \in \Gamma(T^*M)$, by
\begin{equation}\label{curv-C-R}
C_D(\alpha,\beta) = R_D(\alpha,\beta) -
D_{\varphi(\Lambda^{\sharp}(\alpha),\Lambda^{\sharp}(\beta),\cdot)}.
\end{equation}
Hence, according to Definition \ref{def-pre-quan}, Proposition
\ref{br-curv}, and the formul{\ae} (\ref{phi-bracket}) and
(\ref{curv-C-R}), we conclude that $(M,\Lambda)$ is prequantizable
as Poisson manifold (\cite{vai}) if and only if
$(M,\Lambda,\varphi)$ is prequantizable as twisted Poisson manifold.

\vspace{2mm}

\noindent \emph{2) Twisted Poisson manifolds associated to
symplectic manifolds I:} Any twisted Poisson structure
$(\Lambda,\varphi)$ on a $2n$-dimensional differentiable manifold
$M_0$, $n\geq2 $, constructed by a symplectic structure $\omega_0$
on $M_0$ as in Example 2 of the subsection
\ref{examples-tw-Poisson}, i.e. $\Lambda=f\Lambda_0$ and
$\varphi=-f^{-2}\omega_0\wedge df$, where
$\Lambda_0=\Lambda_0^{\sharp}(\omega_0)$ and $f$ is an arbitrary
non constant function on $M_0$, is not prequantizable. We will
prove that the prequantization equation (\ref{cond-quant}) has not
solutions on $M_0$. We note that every vector field $Z$ on $M_0$
can be written as $Z=\Lambda_0^{\sharp}(\alpha)$ with $\alpha\in
\Gamma(T^*M)$. Therefore,
\begin{equation}\label{eq-ex-2}
\Lambda +\partial_{\varphi}Z = \Lambda_0^{\sharp}(f\omega_0 -f
d\alpha -\alpha \wedge df).
\end{equation}
On the other hand, if there exists a closed $2$-form $\Phi$ on
$M_0$ such that, for a particular vector field $Z$ on $M_0$,
$\Lambda + \partial_{\varphi}Z =
\Lambda^{\sharp}(\Phi)=f^2\Lambda_0^{\sharp}(\Phi)$, then, taking
into account (\ref{eq-ex-2}) and the fact that
$\Lambda_0^{\sharp}$ is inversible, we will must have
$$
f^2\Phi =f\omega_0 -f d\alpha -\alpha \wedge df \Leftrightarrow \Phi
= f^{-1}\omega_0 -f^{-1}d\alpha -f^{-2}\alpha \wedge df.
$$
But, in this case, $d\Phi = -f^{-2}\omega_0\wedge df = \varphi
\neq 0$, for any non constant function $f$ on $M_0$. Thus,
$(M_0,\Lambda,\varphi)$ is not prequantizable.

\vspace{2mm}

\noindent \emph{3) Twisted Poisson manifolds associated to
symplectic manifolds II:} Let $(M,\Lambda,\varphi)$ be a twisted
Poisson manifold constructed by a symplectic manifold
$(M_0,\omega_0)$ as in Example 3 of the subsection
\ref{examples-tw-Poisson}, i.e., $M=M_0\times \R$,
$$
\Lambda = e^t(\Lambda_0 + \Lambda_0^{\sharp}(df)\wedge
\frac{\partial}{\partial t}) \quad \mathrm{and} \quad \varphi =
-e^{-t}\omega_0\wedge dt,
$$
$t$ being the canonical coordinate on $\R$ and $f\in
C^{\infty}(M_0,\R)$. We assume that the symplectic structure
$\omega_0$ is of the particular type $\omega_0 = d\alpha_0
-\alpha_0\wedge df$, where $\alpha_0$ is a convenient $1$-form on
$M_0$, i.e., $\alpha_0$ is a $1$-form on $M_0$ such that $\omega_0
= d\alpha_0 -\alpha_0\wedge df$ is nondegenerate and
$d\omega_0=-d\alpha_0\wedge df =0$. Then, $(M,\Lambda,\varphi)$ is
prequantizable. Effectively, if we take $Z=\partial / \partial t$
and $\Phi =d(e^{-t}\alpha_0)$, which represents the integral
cohomology class $[0]\in H_{dR}^2(M,\R)$ of $M$, after a simple
computation we obtain that (\ref{cond-quant}) holds on $M$.

\vspace{2mm}

\noindent \emph{4) Exact twisted Poisson manifolds:} Let
$(M,\Lambda,\varphi)$ be an exact twisted Poisson manifold,
namely, there exists a vector field $X$ on $M$ such that $\Lambda
=
\partial_{\varphi}X$, fact that is equivalent to
$[\Lambda]^{\varphi}=[0]^{\varphi}\in H_{L-tP}^2(M)$. Then,
$(M,\Lambda,\varphi)$ is prequantizable. The vector field $Z=-X$ and
the $2$-form $\Phi = 0$ satisfy the prequantization condition
(\ref{cond-quant}). The trivial complex line bundle $\pi : M\times
\mathbb{C} \to M$, whose space of global cross sections
$\Gamma(M\times \mathbb{C})$ is equal to the set
$C^{\infty}(M,\mathbb{C})$, equipped with the usual Hermitian metric
$h$, i.e., for any $s_1,s_2 \in C^{\infty}(M,\mathbb{C})$,
$h(s_1,s_2) = s_1\bar{s}_2$, and the compatible with $h$
contravariant derivative $D$ given, for any $\alpha \in
\Gamma(T^*M)$ and $s\in C^{\infty}(M,\mathbb{C})$, by $D_{\alpha}s =
\Lambda^{\sharp}(\alpha)s$, is a prequantization bundle of
$(M,\Lambda,\varphi)$.

\vspace{2mm}

\noindent \emph{5) Twisted Poisson structures induced by twisted
Jacobi manifolds:} Let $(M_0,\Lambda_0, d\omega_0)$ be the twisted
Poisson manifold constructed by a twisted Jacobi manifold
$(M,\Lambda,E,\varpi^*\omega_0)$ in Example 4 of the subsection
\ref{examples-tw-Poisson}. Let $\eta$ be the $1$-form along $M_0$
that verifies $i(E)\eta=1$ and $i(X)\eta =0$, for any vector field
$X$ on $M$ tangent to $M_0$. By integration along the integral
curves of $E$ and by restriction, if necessary, of the tubular
neighbourhood $U$ of $M_0$ in $M$, we can construct a function $h$
on $M$ such that $h\vert_{M_0}=0$ and $i(E)dh =1$, hence
$dh\vert_{M_0}=\eta$. Let $X_h = \Lambda^{\sharp}(dh)+hE$ be the
Hamiltonian vector field of $h$ with respect the twisted Jacobi
structure $(\Lambda,E)$ on $M$. Since $[E,\Lambda](dh,\cdot)=0$,
we have $[E,X_h]=E$, whence we conclude that $X_h$ is projectable
along the integral curves of $E$ onto $M_0$. Let $Z_0$ be its
projection, i.e. $Z_0 = \varpi_*X_h =
\varpi_*(\Lambda^{\sharp}(dh)) = \Lambda^{\sharp}(\eta)$. The
differential operator of first order $X_h-1$ verifies (see in
\cite{jf} Propositions 3.1 and 3.5), for any $f,g\in \C$, the
relation
\begin{eqnarray*}
(X_h-1)\{f,g\} & = & \{h,\{f,g\}\} \\
& = & \{\{h,f\},g\} + \{f,\{h,g\}\} \\
&  & +\,\Lambda^{\sharp}(\varpi^*d\omega_0)(dh,df,dg) +
\Lambda^{\sharp}(\varpi^*\omega_0)\wedge E (dh,df,dg).
\end{eqnarray*}
By projection, we obtain that the first order differential
operator $Z_0-1$ verifies, for any $f_0,g_0 \in
C^{\infty}(M_0,\mathbb{R})$,
\begin{eqnarray*}
(Z_0-1)\{f_0,g_0\} & = & \{(Z_0-1)f_0,g_0\} + \{f_0,(Z_0-1)g_0\} \\
& & -\,
d\omega_0(Z_0,\Lambda_0^{\sharp}(df_0),\Lambda_0^{\sharp}(dg_0)) +
\omega_0(\Lambda_0^{\sharp}(df_0),\Lambda_0^{\sharp}(dg_0)).
\end{eqnarray*}
From the above relation, after a simple computation, we get
$$
\Lambda_0 + \partial_{d\omega_0}(-Z_0) =
\Lambda_0^{\sharp}(\omega_0).
$$
If $\omega_0$ is closed and represents an integral cohomology
class of $M_0$, then, the last equation means that the induced
twisted Poisson structure $(\Lambda_0,d\omega_0) = (\Lambda_0, 0)$
is a prequantizable Poisson structure on $M_0$. (For more details,
see \cite{pt}.)

\vspace{2mm}

\noindent \emph{6) A $r$-matrix type twisted Poisson structure:}
We consider the twisted Poisson structure of the Example 4.8 in
\cite{ksml} (see, also Example 5 in \cite{ksc}). Let $\mathcal{G}$
be the subalgebra of the Lie algebra of $GL(3,\mathbb{R})$ spanned
by $\{e_{ij}\, / \, 1\leq i \leq 2, \; 1\leq j \leq 3\}$. We
denote by $\{e_{ij}^*\, / \, 1\leq i \leq 2, \; 1\leq j \leq 3\}$
the dual basis of its dual space $\mathcal{G}^*$. The pair $(r,
\varphi)$, where
$$
r = e_{11}\wedge e_{22} + e_{13}\wedge e_{23} \quad \mathrm{and}
\quad \varphi = - (e_{11}^* + e_{22}^*)\wedge e_{13}^*\wedge
e_{23}^*,
$$
defines a twisted Poisson structure on $\mathcal{G}$. It is easy
to check that $\varphi$ is closed and $\frac{1}{2}[r,r] =
r^{\sharp}(\varphi)$. We will show that $(r, \varphi)$ is not
prequantizable on $\mathcal{G}$. After a simple, but long,
computation, we prove that the space of closed $2$-forms of
$\mathcal{G}$ is spanned by $\{(e_{11}^* - e_{22}^*)\wedge
e_{12}^*,(e_{11}^* - e_{22}^*)\wedge e_{21}^*, e_{12}^*\wedge
e_{21}^*\}$ and that, for any such form $\Phi$ of $\mathcal{G}$,
$r^{\sharp}(\Phi)=0$. On the other hand, for any vector
$Z=\sum_{i,j}\lambda_{ij}e_{ij}$, $\lambda_{ij}\in \mathbb{R}$, of
$\mathcal{G}$, we have
\begin{eqnarray*}
r + \partial_{\varphi}Z & = & - \lambda_{12} e_{11}\wedge e_{12} -
\lambda_{13} e_{11}\wedge e_{13} + \lambda_{21} e_{11}\wedge e_{21}
+ e_{11}\wedge e_{22} + \lambda_{12}e_{12}\wedge e_{22}\\
&  &  + \, (1-\lambda_{12} + \lambda_{21})e_{13}\wedge e_{23} -
\lambda_{21}e_{21}\wedge e_{22} + \lambda_{23}e_{22}\wedge e_{23}
\neq 0.
\end{eqnarray*}
Hence, we conclude that the prequantization equation
(\ref{cond-quant}) has not a solution $(Z,\Phi)$ on $\mathcal{G}$.

\section{Quantization}\label{sec-quant}
The second step of the geometric quantization of a twisted Poisson
manifold $(M,\Lambda,\varphi)$ is the construction of a Hilbert
space out of its prequantization space $\Gamma(K)$ on which a
convenient $\varphi$-twisted Lie subalgebra of the $\varphi$-twisted
Lie algebra $(\C,\{\cdot,\cdot\})$ will be represented irreducibly.
For this reason, we introduce the notion of \emph{polarization}
(\cite{sn}, \cite{wd}) of $(M,\Lambda,\varphi)$ as follows.

\vspace{2mm}

We consider the complexification $T^*M \otimes \mathbb{C}$ of the
cotangent bundle $T^*M$ of $M$ and we endow the space of its cross
sections $\Gamma(T^*M \otimes \mathbb{C})$ with the natural
extension of the bracket (\ref{def-Lie-br}), also denoted by
$\{\cdot , \cdot\}^{\varphi}$. Then, $(T^*M \otimes
\mathbb{C},\{\cdot , \cdot\}^{\varphi}, \Lambda^{\sharp})$, where
$\Lambda^{\sharp} : T^*M \otimes \mathbb{C} \to TM \otimes
\mathbb{C}$ is the natural extension to $T^*M \otimes \mathbb{C}$ of
the vector bundle map given by (\ref{bnd-map}), is a complex Lie
algebroid over $M$, in the sense of \cite{aw}, and $(\Gamma(T^*M
\otimes \mathbb{C}), \{\cdot , \cdot\}^{\varphi})$ is a complex Lie
algebra. We define a \emph{polarization of} $(M,\Lambda,\varphi)$ to
be a complex Lie subalgebra $\mathcal{P}$ of $(\Gamma(T^*M \otimes
\mathbb{C}), \{\cdot , \cdot\}^{\varphi})$ such that, for all
$\alpha, \beta \in \mathcal{P}$,
$$
\Lambda(\alpha,\beta) = 0.
$$
When $\mathcal{P}$ is fixed, we set
$$
P(\mathcal{P})= \{f\in \C \, / \,\{df,\alpha\}^{\varphi}\in
\mathcal{P}, \quad \mathrm{for}\;\mathrm{all}\;\; \alpha \in
\mathcal{P} \}
$$
and we consider the subset $\widetilde{P(\mathcal{P})}$ of
$P(\mathcal{P}) \times P(\mathcal{P})$ given by
\begin{eqnarray*}
\lefteqn{\widetilde{P(\mathcal{P})}=\Big{\{}(f,g) \in
P(\mathcal{P}) \times P(\mathcal{P}) \setminus
\Delta(P(\mathcal{P}) \times
P(\mathcal{P})) \, \big{/}} \\
&  & \quad \quad \quad
\{\varphi(\Lambda^{\sharp}(df),\Lambda^{\sharp}(dg),\cdot
),\alpha\}^{\varphi}\in \mathcal{P}, \quad
\mathrm{for}\;\mathrm{all}\;\; \alpha \in \mathcal{P}\Big{\}},
\end{eqnarray*}
where $\Delta(P(\mathcal{P}) \times P(\mathcal{P}))$ denotes the
diagonal of $P(\mathcal{P}) \times P(\mathcal{P})$. Clearly,
$\widetilde{P(\mathcal{P})}$ is symmetric with respect to
$\Delta(P(\mathcal{P}) \times P(\mathcal{P}))$. If
$\mathcal{Q}(\mathcal{P})$ is the projection of
$\widetilde{P(\mathcal{P})}$ on $P(\mathcal{P})$, we have that
$(\mathcal{Q}(\mathcal{P}), \{\cdot,\cdot\})$ is a
$\varphi$-twisted Lie subalgebra of $(\C, \{\cdot,\cdot\})$ which
will be called the subalgebra of the \emph{straightforwardly
quantizable observables of} $(M,\Lambda,\varphi)$. Obviously, if
$\varphi =0$, from the above definitions we obtain those given in
\cite{vai} for Poisson manifolds.

\vspace{2mm}

Now, in order to build a Hilbert space out of $\Gamma(K)$ on which
the quantum operators corresponding to the elements of
$\mathcal{Q}(\mathcal{P})$ act, we apply a classical method in
geometric quantization using the line bundle of complex
half-densities of $M$.

\vspace{2mm}

Let $\mathcal{D}$ be the \emph{half-density bundle} associated to
$TM$. It is well known (\cite{bw}, \cite{sbg}, \cite{ya}) that its
cross sections $\varrho$, called \emph{half-densities of} $M$, are
complex valued maps defined on the set $\mathcal{B}(TM)$ of basis of
$\Gamma(TM)$ such that, for any $x\in M$, $e_x\in \mathcal{B}(T_xM)$
and $A_x\in GL(T_xM)$,
$$
\varrho_x(e_xA_x)=\varrho_x(e_x)|\det A_x|^{1/2}.
$$
Since $GL(T_xM)$ acts transitively on $\mathcal{B}(T_xM)$,
$\varrho_x$ is determined by its value on a single basis of
$\Gamma(T_xM)$. As a result, we have that $\mathcal{D}$ is a
complex line bundle over $M$ which is defined by the transition
functions that are the square roots of the absolute values of the
Jacobians of the coordinate transformations $\tilde{x}_i =
\tilde{x}_i(x_j)$, i.e. $|\partial x_j /\partial
\tilde{x}_i|^{1/2}$. The Lie derivatives $\mathcal{L}$ of
$\varrho$ are defined as for tensors fields on $M$, (see,
\cite{ya}).

\vspace{2mm}

We assume that $(M,\Lambda,\varphi)$ is a prequantizable twisted
Poisson manifold. Let $\pi : K \to M$ be its prequantization bundle,
$h$ the Hermitian metric on $\pi : K \to M$ and $D$ a compatible
with $h$ contravariant derivative on $\pi : K \to M$ whose curvature
$C_D$ verifies (\ref{curv-Lambda}). Using the properties
(\ref{lin-D}), (\ref{rel-contr-der}) of $D$ and those of
$\mathcal{L}$, we can extend $D$ to a mapping, also denoted by $D$,
$$
D : \Gamma(T^*M\otimes \mathbb{C}) \to
\mathrm{End}_{\mathbb{C}}(\Gamma(K\otimes \mathcal{D}))
$$
by putting, for any $\alpha \in \Gamma(T^*M\otimes \mathbb{C})$
and $s\otimes \varrho \in \Gamma(K\otimes \mathcal{D})$,
\begin{equation}\label{ext-D}
D_{\alpha}(s\otimes \varrho) = D_{\alpha}s\otimes \varrho +
s\otimes \mathcal{L}_{\Lambda^{\sharp}(\alpha)}\varrho.
\end{equation}
Therefore, the representation $\,\widehat{}\, :
(\C,\{\cdot,\cdot\}) \to \mathrm{End}_{\mathbb{C}}(\Gamma(K))$
given by (\ref{def-repres}) can be extended to a representation of
$(\C,\{\cdot,\cdot\})$ on $\Gamma(K\otimes \mathcal{D})$, also
denoted by $\,\widehat{}\,$, by setting, for all $f\in \C$ and
$s\otimes \varrho \in \Gamma(K\otimes \mathcal{D})$,
\begin{equation}\label{ext-repres}
\hat{f}(s\otimes \varrho) = D_{df}(s\otimes \varrho) + 2\pi i f
(s\otimes \varrho).
\end{equation}
Because of (\ref{ext-D}), (\ref{ext-repres}) can be written as
\begin{equation}\label{ext-f}
\hat{f}(s\otimes \varrho) = (\hat{f}(s))\otimes \varrho + s\otimes
\mathcal{L}_{\Lambda^{\sharp}(df)}\varrho.
\end{equation}
Thus, taking into account (\ref{ext-f}), (\ref{homo-bracket}),
(\ref{def-Lie-br}), the property of the anchor map
$\Lambda^{\sharp}$ and that of the Lie derivative, we can easily
check that the prequantization condition (\ref{homo-bracket})
remains true, i.e., for any $f,g \in \C$ and $s\otimes \varrho \in
\Gamma(K\otimes \mathcal{D})$,
$$
\widehat{\{f,g\}}(s\otimes \varrho) =
[\hat{f},\hat{g}]^{\varphi}(s\otimes \varrho).
$$
Furthermore, applying (\ref{ext-repres}), (\ref{ext-D}),
(\ref{eq-curv}) and (\ref{curv-Lambda}), we deduce that
\begin{equation}\label{pr-D-ext}
D_{\alpha}(\hat{f}(s\otimes \varrho)) =
\hat{f}(D_{\alpha}(s\otimes \varrho)) -
D_{\{df,\alpha\}^{\varphi}}(s\otimes \varrho),
\end{equation}
for all $\alpha \in \Gamma(T^*M \otimes \mathbb{C})$, $f\in \C$
and $s\otimes \varrho \in \Gamma(K\otimes \mathcal{D})$.

Now, for a fixed polarization $\mathcal{P}$ of
$(M,\Lambda,\varphi)$, we consider the subset $\mathcal{H}_0$ of
$\Gamma(K\otimes \mathcal{D})$ given by
\begin{equation}\label{def-H0}
\mathcal{H}_0 = \{s\otimes \varrho \in \Gamma(K\otimes \mathcal{D})
\, / \, D_{\alpha}(s\otimes \varrho) = 0, \quad \mathrm{for}\;
\mathrm{all}\; \alpha \in \mathcal{P}\},
\end{equation}
and we assume that $\mathcal{H}_0 \neq \{0\}$, which is a
\emph{Bohr-Sommerfeld type condition} (see, \cite{sn}). We have
that, for any $f \in \mathcal{Q}(\mathcal{P})$ and $s\otimes \varrho
\in \mathcal{H}_0$, $\hat{f}(s\otimes \varrho)\in \mathcal{H}_0$. In
fact, for every $\alpha \in \mathcal{P}$, $\{df,\alpha
\}^{\varphi}\in \mathcal{P}$ and $D_{\alpha}(s\otimes \varrho)=0$.
Hence, according to (\ref{pr-D-ext}), we get
$$
D_{\alpha}(\hat{f}(s\otimes \varrho)) = \hat{f}(D_{\alpha}(s\otimes
\varrho)) - D_{\{df,\alpha\}^{\varphi}}(s\otimes \varrho) =
\hat{f}(0) - 0= 0,
$$
which means that $\hat{f}(s\otimes \varrho) \in \mathcal{H}_0$.
Consequently, $\hat{f}\vert_{\mathcal{H}_0} : \mathcal{H}_0 \to
\mathcal{H}_0$ is well defined for any $f\in
\mathcal{Q}(\mathcal{P})$. Thus, $\mathcal{H}_0 $ can be viewed as
a \emph{quantization space for} $\mathcal{Q}(\mathcal{P})$.

\vspace{2mm}

Next, we distinguish the following cases.

\vspace{2mm}

If $M$ is \emph{compact}, then, $\mathcal{H}_0$ equipped with the
inner product
\begin{equation}\label{inner}
\langle s_1\otimes \varrho_1, s_2\otimes \varrho_2 \rangle =
\int_M h(s_1,s_2)\varrho_1 \bar{\varrho}_2,
\end{equation}
$h$ being the Hermitian metric on $\pi : K \to M$ and bar denoting
the complex conjugation, is a \emph{pre-Hilbert space}. Moreover,
the operators $\hat{f}$ defined by (\ref{ext-repres}) or
(\ref{ext-f}) are anti-Hermitian with respect to (\ref{inner}).
This results as follows:
\begin{eqnarray}\label{anti-Herm}
\lefteqn{\langle \hat{f}(s_1\otimes \varrho_1), s_2\otimes
\varrho_2 \rangle + \langle s_1\otimes \varrho_1,
\hat{f}(s_2\otimes
\varrho_2) \rangle \stackrel{(\ref{ext-f})}{=}} \nonumber \\
& \langle (\hat{f}s_1)\otimes \varrho_1 + s_1\otimes
\mathcal{L}_{\Lambda^{\sharp}(df)}\varrho_1, s_2\otimes \varrho_2
\rangle + \langle s_1\otimes \varrho_1, (\hat{f}s_2)\otimes
\varrho_2 + s_2\otimes
\mathcal{L}_{\Lambda^{\sharp}(df)}\varrho_2\rangle \stackrel{(\ref{inner})}{=} \nonumber \\
\lefteqn{\int_M \Big( \big( h(\hat{f}s_1,s_2) +
h(s_1,\hat{f}s_2)\big)\varrho_1 \bar{\varrho}_2 + h(s_1,s_2)\big(
(\mathcal{L}_{\Lambda^{\sharp}(df)}\varrho_1)\bar{\varrho}_2 +
\varrho_1(\mathcal{L}_{\Lambda^{\sharp}(df)}\bar{\varrho}_2)
\big)\Big) \stackrel{(\ref{def-hermitian})}{=}} \nonumber \\
\lefteqn{\int_M \Big( \Lambda^{\sharp}(df)\big(
h(s_1,s_2)\big)\varrho_1 \bar{\varrho}_2 + h(s_1,s_2)\big(
(\mathcal{L}_{\Lambda^{\sharp}(df)}\varrho_1)\bar{\varrho}_2 +
\varrho_1(\mathcal{L}_{\Lambda^{\sharp}(df)}\bar{\varrho}_2)
\big)\Big) =} \nonumber \\
 \lefteqn{\int_M
\mathcal{L}_{\Lambda^{\sharp}(df)}\big(h(s_1,s_2)\varrho_1
\bar{\varrho}_2\big) = 0,}
\end{eqnarray}
where the last equality is true because of the density version of
Stokes' Theorem (\cite{vai-tor}, \cite{sbg}). If we require the
quantization space for $\mathcal{Q}(\mathcal{P})$ is a Hilbert
space, we take the completion $\mathcal{H}$ of $\mathcal{H}_0$. In
order to obtain Hermitian operators on $\mathcal{H}$, we prolong
$\hat{f}$ on $\mathcal{H}$ so that the obtained operators are
anti-Hermitian and then we multiple these by $i$. Then, condition
(\ref{homo-bracket}) is true up to the constant factor $i$.

\vspace{2mm}

If $M$ is \emph{not compact}, we consider the subalgebra
$\mathcal{P}_0$ of $(\Gamma(T^*M), \{\cdot , \cdot\}^{\varphi})$
whose complexification is $\mathcal{P}\cap \bar{\mathcal{P}}$ (so,
for all $\alpha, \beta \in \mathcal{P}_0$,
$\Lambda(\alpha,\beta)=0$) and we postulate
$\Lambda^{\sharp}(\mathcal{P}_0)$ to defines a regular foliation
$\mathcal{F}$ of $M$ whose the leaf space $N=M/\mathcal{F}$ is a
Hausdorff manifold. We can easily show that, for any $f\in
\mathcal{Q}(\mathcal{P})$ and $\alpha \in \mathcal{P}_0$,
$\{df,\alpha \}^{\varphi} \in \mathcal{P}_0$. Therefore, for any
$f\in \mathcal{Q}(\mathcal{P})$, the Hamiltonian vector field
$\Lambda^{\sharp}(df)$ is projectable with respect to
$\Lambda^{\sharp}(\mathcal{P}_0)$ onto $N$ (we have, for all
$\alpha \in \mathcal{P}_0$,
$[\Lambda^{\sharp}(df),\Lambda^{\sharp}(\alpha)]=\Lambda^{\sharp}(\{df,\alpha
\}^{\varphi})\in \Lambda^{\sharp}(\mathcal{P}_0)$). Also, if
$\varpi : M \to N$ denotes the canonical projection of $M$ onto
$N$, we have
\begin{eqnarray}\label{f-on-N}
\hat{f}(s\otimes \varpi^*\varrho_{_N}) & = &(\hat{f}s)\otimes
\varpi^*\varrho_{_N} + s\otimes
\mathcal{L}_{\Lambda^{\sharp}(df)}(\varpi^*\varrho_{_N}) \nonumber \\
& =& (\hat{f}s)\otimes \varpi^*\varrho_{_N} + s\otimes
\varpi^*(\mathcal{L}_{\varpi_*(\Lambda^{\sharp}(df))}\varrho_{_N}),
\end{eqnarray}
for all $s\in \Gamma(K)$ and $\varrho_{_N}$ a complex half-density
of $N$. The last equality permits us, instead of using arbitrary
half-densities of $M$ for the construction of $\mathcal{H}_0$, to
use $\mathcal{F}$-transversal half-densities of $M$ that are the
pull-back under $\varpi$ of half-densities of $N$. Then, for any
$\alpha \in \mathcal{P}_0$ and $\varrho_{_N}$ complex half-density
of $N$,
$\mathcal{L}_{\Lambda^{\sharp}(\alpha)}(\varpi^*\varrho_{_N})=0$.
Using this fact, (\ref{def-hermitian}), (\ref{ext-D}) and
(\ref{def-H0}), we have that, for all $s_1\otimes
\varpi^*\varrho_{1_N}, s_2\otimes \varpi^*\varrho_{2_N} \in
\mathcal{H}_0$ and $\alpha \in \mathcal{P}_0$,
$$
\mathcal{L}_{\Lambda^{\sharp}(\alpha)}\big(h(s_1,s_2)\varpi^*\varrho_{1_N}\varpi^*\bar{\varrho}_{2_N}
\big) = 0,
$$
which means that
$h(s_1,s_2)\varpi^*\varrho_{1_N}\varpi^*\bar{\varrho}_{2_N}$ can be
projected to a complex $1$-density $\delta_N$ of $N$ (the
multiplication of two half-densities yields a $1$-density). Hence,
$\mathcal{H}_0$ can be replaced by its subspace $\mathcal{H}_0^c$
formed by the sections that are projectable to $N$ and whose
projection has as support a compact subset of $N$. In general, we
may expect that $\mathcal{H}_0^c \neq \{0\}$. In this case,
$\mathcal{H}_0^c$ endowed with the inner product
$$
\langle s_1\otimes \varpi^*\varrho_{1_N}, s_2\otimes
\varpi^*\varrho_{2_N}\rangle = \int_N \delta_N
$$
is a pre-Hilbert space. Furthermore, working as in
(\ref{anti-Herm}), we prove that, for any $f\in
\mathcal{Q}(\mathcal{P})$, the corresponding operator $\hat{f}$
verifies
$$
\langle \hat{f}(s_1\otimes \varpi^*\varrho_{1_N}), s_2\otimes
\varpi^*\varrho_{2_N} \rangle + \langle s_1\otimes
\varpi^*\varrho_{1_N}, \hat{f}(s_2\otimes \varpi^*\varrho_{2_N})
\rangle = \int_N \mathcal{L}_{\Lambda^{\sharp}(df)}\delta_N = 0,
$$
whence we deduce the anti-Hermitian character of $\hat{f}$. In
order that the quantization space of $\mathcal{Q}(\mathcal{P})$ be
a Hilbert space and in order to obtain Hermitian operators on this
space, we proceed as in the compact case.

\subsection{Example}
Below, we will study the quantization of the prequantizable
twisted Poisson manifold $(M,\Lambda,\varphi)$ presented in
Example 3 of paragraph \ref{ex-prequan-tw-Poisson}.

We have $(M,\Lambda,\varphi) = (M_0\times \mathbb{R},\;
e^t(\Lambda_0 + \Lambda_0^{\sharp}(df)\wedge
\frac{\partial}{\partial t}),\; -e^{-t}\omega_0\wedge dt)$, where
$(M_0,\omega_0) = (M_0, \; d\alpha_0 -\alpha_0\wedge df)$ is a
symplectic manifold, with $\alpha_0\in \Gamma(T^*M_0)$ and $f\in
\C$, $\Lambda_0 = \Lambda_0^{\sharp}(\omega_0)$ and $t$ is the
canonical coordinate on $\mathbb{R}$. As we have seen, a solution
of (\ref{cond-quant}) is $(Z,\Phi) = (\partial / \partial t,
\,d(e^{-t}\alpha_0))$, therefore, the prequantization bundle of
$(M,\Lambda,\varphi)$ is the trivial complex line bundle $\pi :
M\times \mathbb{C} \to M$ equipped with the usual Hermitian metric
$h$ and the Hermitian contravariant derivative $D$ defined, for
any $\alpha \in \Gamma(T^*M)$ and $s\in \Gamma(M\times \mathbb{C})
= \c$, by
\begin{equation}\label{example-D}
D_{\alpha}s = \Lambda^{\sharp}(\alpha)s.
\end{equation}

We take $M_0 = \mathbb{R}^{2n}$, $n\geq 2$. Let $(x_1,x_2, \ldots,
x_{2n})$ be a local coordinates system of $M_0$ in which $\omega_0 =
d\alpha_0 -\alpha_0\wedge df$ has the Darboux's expression, i.e.,
$$
\omega_0 = \sum_{k=1}^n dx_{2k-1}\wedge dx_{2k}.
$$
Hence,
$$
\Lambda = e^t \big(\sum_{k=1}^n \frac{\partial}{\partial
x_{2k-1}}\wedge \frac{\partial}{\partial x_{2k}} + \sum_{k=1}^n
(\frac{\partial f}{\partial x_{2k-1}}\frac{\partial}{\partial
x_{2k}} - \frac{\partial f}{\partial
x_{2k}}\frac{\partial}{\partial x_{2k-1}})\wedge
\frac{\partial}{\partial t}\big)
$$
and
$$
\varphi = -e^{-t}\big(\sum_{k=1}^n dx_{2k-1}\wedge dx_{2k}\wedge
dt \big).
$$
Using the identifications $M = \mathbb{R}^{2n}\times \mathbb{R}\cong
\mathbb{C}^n \times \mathbb{R}$, $z_k = x_{2k-1} + i x_{2k}$ and
$\bar{z}_k = x_{2k-1} - i x_{2k}$, $k =1,\ldots, n$, which give us
$dx_{2k-1} = \frac{1}{2}(dz_k + d\bar{z}_k)$, $dx_{2k} =
-\frac{i}{2}(dz_k - d\bar{z}_k)$, $\frac{\partial}{\partial
x_{2k-1}} = \frac{\partial}{\partial z_k} + \frac{\partial}{\partial
\bar{z}_k}$ and $\frac{\partial}{\partial x_{2k}} =
i(\frac{\partial}{\partial z_k} - \frac{\partial}{\partial
\bar{z}_k})$, we obtain that, in the complex coordinates $(z_1,
\ldots, z_n, t)$ of $M$, the pair $(\Lambda, \varphi)$ is written as
follows:
$$
\Lambda = -2ie^t \big(\sum_{k=1}^n \frac{\partial}{\partial
z_k}\wedge \frac{\partial}{\partial \bar{z}_k} + \sum_{k=1}^n
(\frac{\partial f}{\partial z_k}\frac{\partial}{\partial
\bar{z}_k} - \frac{\partial f}{\partial
\bar{z}_k}\frac{\partial}{\partial z_k})\wedge
\frac{\partial}{\partial t}\big),
$$
$$
\varphi = -\frac{i}{2}e^{-t}\big(\sum_{k=1}^n dz_k\wedge
d\bar{z}_k\wedge dt \big).
$$
We observe that a convenient polarization of $(M,\Lambda,\varphi)$
is $\mathcal{P} = \mathrm{span}\{dz_1,\ldots,dz_n\}$. Then, the
set $P(\mathcal{P})$ consists of the functions $g\in \C$ for that
$\{dg, dz_k\}^{\varphi} \in \mathcal{P}$, for any $k = 1,
\ldots,n$. After a computation, we get that the coefficient of
$dt$ in $\{dg, dz_k\}^{\varphi}$ is annihilated. Thus, $\{dg,
dz_k\}^{\varphi} \in \mathcal{P}$ if, and only if, its
coefficients of $d\bar{z}_l$, $l = 1, \ldots,n$, are annihilated,
i.e.,
\begin{equation}\label{system-dif-eq-1}
\frac{\partial}{\partial \bar{z}_l}(- \frac{\partial g}{\partial
\bar{z}_k} + \frac{\partial f}{\partial \bar{z}_k}\frac{\partial
g}{\partial t}) + \frac{\partial g}{\partial
\bar{z}_l}\frac{\partial f}{\partial \bar{z}_k} - \frac{\partial
f}{\partial \bar{z}_k}\frac{\partial f}{\partial
\bar{z}_l}\frac{\partial g}{\partial t} = 0, \quad \quad \forall
\, l =1,\ldots,n.
\end{equation}
Now, we consider the set $\widetilde{P(\mathcal{P})}\subset
P(\mathcal{P})\times P(\mathcal{P})$ of the pairs $(g_1,g_2)$ of
different solutions of the system (\ref{system-dif-eq-1}) for that
$\{\varphi(\Lambda^{\sharp}(dg_1),\Lambda^{\sharp}(dg_2),\cdot),\,
dz_k\}^{\varphi} \in \mathcal{P}$, for any $k = 1, \ldots,n$, and
we take its projection $\mathcal{Q}(\mathcal{P})$ on
$P(\mathcal{P})$. The set $\mathcal{Q}(\mathcal{P})$ is the one of
straightforwardly quantizable observables of
$(M,\Lambda,\varphi)$. We note that a solution of
(\ref{system-dif-eq-1}) is $g_1 = f+t$. Since
$\Lambda^{\sharp}(dg_1)=0$,
$$
\{\varphi(\Lambda^{\sharp}(dg_1),\Lambda^{\sharp}(dg_2),\cdot),\,
dz_k\}^{\varphi} = \{\varphi(0,\Lambda^{\sharp}(dg_2),\cdot),\,
dz_k\}^{\varphi} =\{0,\, dz_k\}^{\varphi}=0 \in \mathcal{P},
$$
for any other $g_2 \in P(\mathcal{P})$ and any $dz_k$,
$k=1,\ldots,n$. So, $f+t \in \mathcal{Q}(\mathcal{P})$.

Next, we have to determine the corresponding quantization space
$\mathcal{H}_0$ for $\mathcal{Q}(\mathcal{P})$. The bundle
$\mathcal{D}$ of complex half-densities over $M=\mathbb{C}^n
\times \mathbb{R}$ is also trivial and it has a basis that can be
written formally as $\beta = |v|^{1/2}$, where
$$
v = dx_1 \wedge \ldots \wedge dx_{2n} \wedge dt = (\frac{i}{2})^n
dz_1 \wedge \ldots \wedge dz_n \wedge d\bar{z}_1 \wedge \ldots
\wedge d\bar{z}_n \wedge dt.
$$
Hence, taking $1$ as the unitary basis of $K=M\times \mathbb{C}$,
we have that any section $s\otimes \varrho$ of $K\otimes
\mathcal{D}$ can be written as $s\otimes \varrho = 1\otimes
(\chi\beta)$, where $\chi \in \c$. Let $D$ be the extension
(\ref{ext-D}) of the Hermitian contravariant derivative on $\pi :
K \to M$ given by (\ref{example-D}). Then, according to the
formula $\mathcal{L}_X\beta = \frac{1}{2}(\mathrm{div}X)\beta$
presented in \cite{vai-tor} (see, also \cite{ya}), we get
$$
D_{dz_k}(1\otimes (\chi\beta)) = 1\otimes
\mathcal{L}_{\Lambda^{\sharp}(dz_k)}(\chi\beta) = 1\otimes
(\mathcal{L}_{\Lambda^{\sharp}(dz_k)}\chi +
\frac{\chi}{2}\mathrm{div}\Lambda^{\sharp}(dz_k))\beta = 0
$$
if, and only if,
\begin{equation}\label{eq-chi}
\mathcal{L}_{\Lambda^{\sharp}(dz_k)}\chi +
\frac{\chi}{2}\mathrm{div}\Lambda^{\sharp}(dz_k) = 0.
\end{equation}
But,
$$
\Lambda^{\sharp}(dz_k) = -2ie^t(\frac{\partial}{\partial
\bar{z}_k}-\frac{\partial f}{\partial
\bar{z}_k}\frac{\partial}{\partial t}) \quad \quad \mathrm{and}
\quad \quad \mathrm{div}\Lambda^{\sharp}(dz_k) =
2ie^t\frac{\partial f}{\partial \bar{z}_k}.
$$
Thus, (\ref{eq-chi}) is equivalent to
$$
-\frac{\partial \chi}{\partial \bar{z}_k} + \frac{\partial
f}{\partial \bar{z}_k}\frac{\partial \chi}{\partial t} +
\frac{\chi}{2}\frac{\partial f}{\partial \bar{z}_k}=0,
$$
whose two solutions are the functions $\chi = e^{\frac{1}{2}f}$ and
$\chi = e^{\frac{1}{2}t}$. Consequently, the quantization space
$\mathcal{H}_0$ is
$$
\mathcal{H}_0 = \{1\otimes (\chi\beta)\in \Gamma(K\otimes
\mathcal{D}) \, / \, -\frac{\partial \chi}{\partial \bar{z}_k} +
\frac{\partial f}{\partial \bar{z}_k}\frac{\partial \chi}{\partial
t} + \frac{\chi}{2}\frac{\partial f}{\partial \bar{z}_k}=0, \;\;
\forall \, k=1, \ldots,n\} \neq \{0\}.
$$
For the elements of $\mathcal{H}_0$ and for $g\in
\mathcal{Q}(\mathcal{P})$, taking into account (\ref{def-repres})
and (\ref{ext-f}), we obtain the quantum operator
$$
\hat{g}(1\otimes (\chi\beta)) = \big(2\pi i g \chi +
\Lambda(dg,d\chi)+
\frac{\chi}{2}\mathrm{div}\Lambda^{\sharp}(dg)\big)(1\otimes
\beta).
$$
Furthermore, the inner product of two elements $1\otimes
(\chi_1\beta), 1\otimes (\chi_2\beta)$ of $\mathcal{H}_0$ with
compact support is
$$
\langle 1\otimes (\chi_1\beta), 1\otimes (\chi_2\beta)\rangle =
\int_M \chi_1\bar{\chi}_2v.
$$

\end{document}